\documentclass[12pt]{article}
\usepackage[totalwidth=480pt,totalheight=680pt]{geometry}

\usepackage{mathptmx}      
\usepackage{graphicx}
\usepackage{amssymb,amsfonts,amsmath}
\usepackage{xspace}
\usepackage{bm}

\usepackage[numbers,sort&compress]{natbib}

\usepackage{amsthm}
\newtheorem{theorem}{Theorem}
\newtheorem{lemma}[theorem]{Lemma}
\newtheorem{corollary}[theorem]{Corollary}

\theoremstyle{definition}
\newtheorem{definition}{Definition}

\newcommand{\vt}{$3$--Cayley tree\xspace}
\newcommand{\vts}{$3$--Cayley trees\xspace}
\newcommand{\cutT}[2]{\ensuremath{T\left({#1},{#2}\right)}}
\newcommand{\cutTopp}[2]{\ensuremath{\overline{T}\left({#1},{#2}\right)}}
\newcommand{\position}{\ensuremath{\phi}}
\newcommand{\posf}[1]{\ensuremath{\position\left({#1}\right)}}

\newcommand{\posfInv}[1]{\ensuremath{\position^{-1}\!\left({#1}\right)}}
\newcommand{\set}[1]{\left\{#1\right\}}

\begin{document}

\title{Thin Tree Position
}


\author{R. Sean Bowman \and
		Douglas R. Heisterkamp \and
		Jesse Johnson
}




\maketitle

\begin{abstract}
  We introduce a method for creating a special type of tree, a
  \emph{tree position}, from a weighted graph.  Leaves of the tree
  correspond to vertices of the original graph, and the tree edges
  contain information which can be used to partition these vertices.
  By repeatedly applying reducing operations to the tree position
  we arrive at a \emph{thin} tree position, and we show that
  partitions arising from thin tree positions have especially nice
  properties.  The algorithm is based the topological notion of thin
  position for knots and $3$--manifolds and builds on the previously
  defined idea of a Topological Intrinsic Lexicographic Order (TILO).

\end{abstract}

\section{Introduction}
\label{intro}

A number of problems in machine learning, data mining and signal
processing require that one find a partition of a graph with small
boundary weight, but whose subsets have relatively large interior
weights. These graphs may directly represent the given data (as in the
case of social networks) or be derived from vector data via a
heuristic such as $k$--nearest neighbors. In either case, these graphs
often have millions of vertices, making a brute force search for
efficient partitions infeasible. 

In the present paper, we introduce an approach we call \emph{Thin
Tree Position} (TTP) for finding efficient partitions of graphs by
carrying out a more targeted search, using intuition from
three dimensional geometry/topology. This approach builds on the
TILO/PRC (Topological Intrinsic Lexicographic Order / Pinch Ratio
Clustering)  algorithm introduced in~\cite{thingraphs}
and~\cite{sdm13prc}, but improves on a number of deficiencies of this
approach. 
The TILO/PRC approach consists of two steps: First, the TILO algorithm
assigns a linear ordering to the vertices of the graph $G$ and then
progressively improves this ordering with respect to a carefully
chosen metric.  Next, the PRC algorithm picks out consecutive blocks
in the final TILO ordering that make up the partition.

Both TILO and TTP are algorithms that use ideas from low dimensional
topology, specifically the concept of \emph{thin position} of knots
and $3$--manifolds~\cite{Gabai,st:thin}.  As such, these methods rely
on the broad shape of the data set instead of its exact geometry.
Carlsson has argued that algorithms of this type should perform well
for certain problems and can give more information than existing
methods~\cite{Carlsson09}.  We expect that TTP will be useful 
for clustering as well as other applications in topological data
analysis.

The present approach addresses two major issues with TILO/PRC: First,
the linear nature of TILO/PRC forces it to treat the subsets of the
partition at the front and back of the TILO  ordering differently from
the subsets in the middle. As we describe in Section~\ref{sect:tpts},
the TTP algorithm replaces the linear ordering with a trivalent tree
structure that allows all the subsets to be treated consistently. As
we discuss in Section~\ref{sect:compare}, this structure is a natural
generalization of TILO orderings.

Second, the gradient-like TILO step of the TILO/PRC is relatively
constrained. In order to make the search space reasonable, TILO only
looks for ways to shift a single vertex at a time forward or backward
in the ordering. Because TPT uses a tree structure, the equivalent
step in TPT is able to move entire branches, which would be equivalent
to allowing TILO to shift arbitrary subsets of consecutive vertices.
However, as we describe in Section~\ref{sect:reduction}, this tree
structure has a very simple reduction criterion analogous to that of
TILO/PRC.

\section{Thin position trees}
\label{sect:tpts}

Consider an undirected graph $G = (V,E)$ in which every edge $e \in E$
has two distinct vertices and every edge is assigned a positive real
number $w(e)$ called its \emph{edge weight}. In many cases, each edge
weight will simply be 1. The weight of a subset of edges $F \subset E$
will be the sum of the weights of those edges, $w(F) = \sum_{f \in F}
w(f)$. An edge is also specified by the two endpoints $f = (v_i, v_j)$
for $f \in E, v_i, v_j \in V$.  A short hand of the weight function is
used for vertices
\begin{gather*} 
w(v_i,v_j) = 
\begin{cases} w(\ (v_i,v_j)\ )\ \text{ if } (v_i,v_j) \in E\\
0 \text{ if } (v_i,v_j) \notin E\\
\end{cases}
\end{gather*}
and for sets $A \subset V, B \subset V$, 
\begin{gather*} 
w(A, B) = \sum_{v_i \in A, v_j \in B} w(v_i,v_j).
\end{gather*}

\begin{definition} 
A \emph{\vt} \cite{CayleyTrees2012} is  a connected acyclic graph in
which every non leaf vertex has  degree of exactly three.  
We call the leaf vertices \emph{boundary vertices} and the
degree three vertices \emph{interior vertices}.
\end{definition}
A \vt is also know as a \emph{boron} tree \cite{BoronTrees2013} in which the
interior vertices correspond to boron atoms and the boundary vertices
correspond to hydrogen atoms.
The number of interior vertices of a \vt is always two less
than the number of boundary vertices; this can be proved, for example,
using an inductive argument, or from the fact that the Euler
characteristic of a tree is one.

We will use \vts\ to study general graphs as follows:
\begin{definition}
Given a finite graph $G$, possibly with weighted edges, a \emph{tree
position} for $G$ is a pair $(C, \position)$ where $C$ is a \vt 
and $\position$ is a one to one map from the vertices of $G$ onto the
boundary vertices of $C$.  
\end{definition}

Figure~\ref{fig:treeposition} presents example tree positions
$(C,\position)$ and $(C',\position')$ for the graph at the top of the
figure.  The outlined circles represent interior vertices, while the
filled in circles are boundary vertices of a tree position.  In this
figure, the relative positions of the vertices of the graph are
maintained in displaying the boundary vertices of the tree positions.

\begin{figure}[!tb]
  \begin{center}
\includegraphics{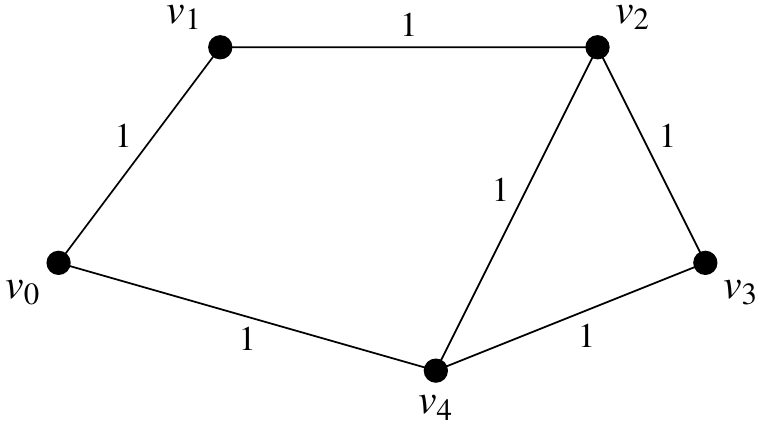}
\vspace*{4ex}

\includegraphics{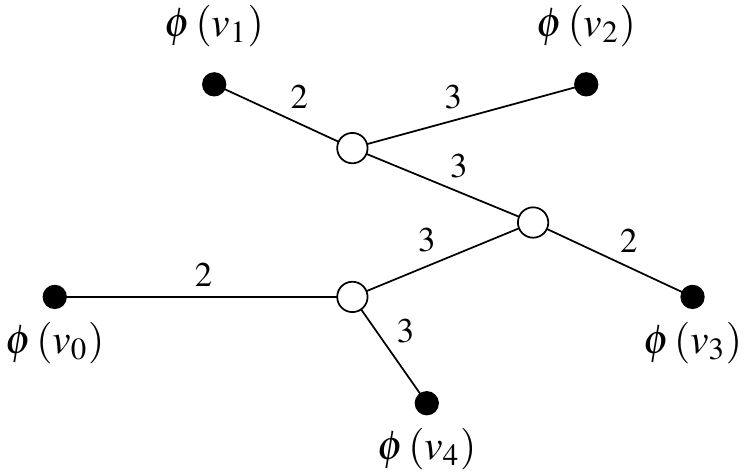}
\hfill
\includegraphics{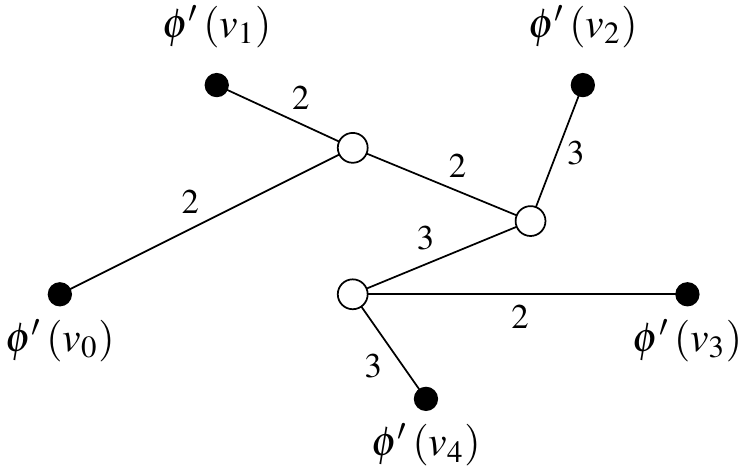}
\vspace*{2ex}
\caption{Graph $G$ with tree positions $(C,\position)$ and $(C',\position')$.}
\label{fig:treeposition}
\end{center}
\end{figure}

\begin{figure}[h]
\begin{center}

\includegraphics{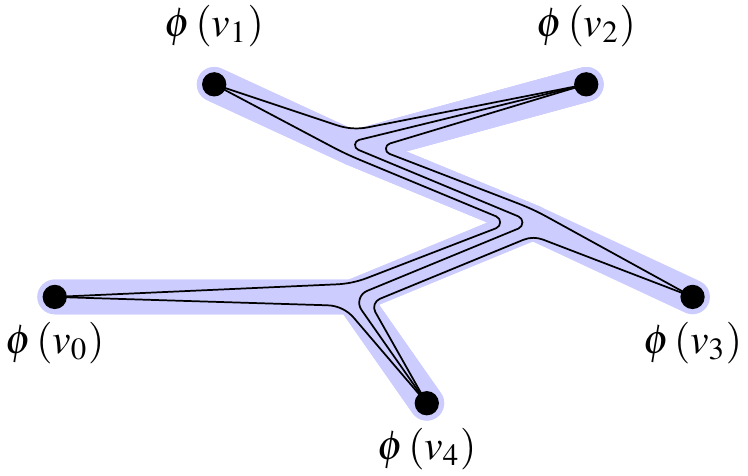}
\hfill
\includegraphics{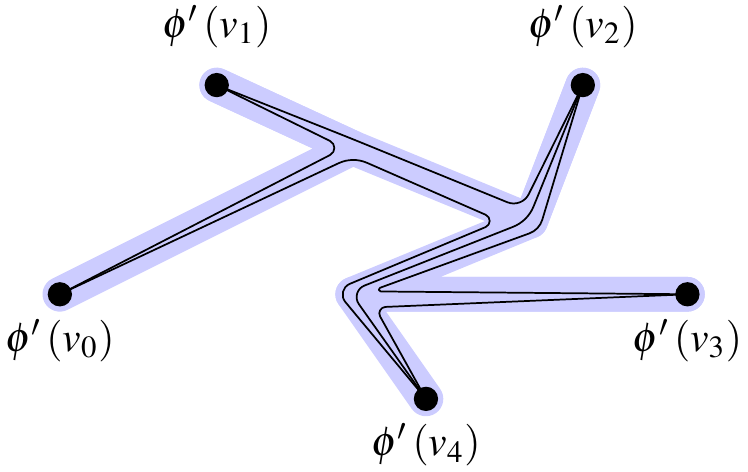}
\vspace*{2ex}
\caption{Edges of $G$ drawn through tree positions $(C,\position)$ and $(C',\position')$.}
\end{center}
\label{fig:topomap}
\end{figure}

Given a tree position $(C, \position)$ for a graph $G$, let $v,w$ be
vertices of $G$ spanned by an edge $e$. Because $C$ is a tree, there
is a unique simple path in $C$ from $\position(v)$ to $\position(w)$. We
say that the edge $e$ \emph{passes through} each of these edges of
$C$. In other words, if $f$ is an edge of $C$ then for any topological
map from the realization of $G$ to the realization of $C$, as in the
picture of Figure~\ref{fig:topomap}, the image of $e$
would be forced to intersect $f$.   For an edge $f$ of $C$, let $R(f)$
be the set of edges of $G$ that pass through $f$.

\begin{definition}
Given a subtree $T$ of a tree position $(C, \position)$ of $G=(V,E)$, the
set of vertices of $V$ that are the image of the boundary vertices in
$T$ is $\posfInv{T} = \left\{ v \in V\ |\ \posf{v} \in T\right\}$.
\end{definition}
\begin{definition}
Let $f$ and $g$ be two distinct edges of a tree position $(C, \position)$.  
Two subtrees of $C$ are defined by cutting the edge $f$.  The subtree
containing $g$ is denoted \cutT{f}{g}, and the subtree not containing $g$ is 
denoted \cutTopp{f}{g}.
\end{definition}
Note that one endpoint of $f$ is in \cutT{f}{g} and the other is in
\cutTopp{f}{g}.  Thus cutting an edge $f$ of a tree position $(C,
\position)$ divides the vertices of $G$ into two disjoint sets:
\posfInv{\cutT{f}{g}} and \posfInv{\cutTopp{f}{g}}.
\begin{definition}
Given a tree position $(C, \position)$ for a graph $G$, the
\emph{width} $b(f)$ of an edge $f$ in $C$ is the sum of the weights of
the edges of $G$ that pass through $f$:  
\begin{align*}
b(f) = w\left(R(f)\right) = w\left( \posfInv{\cutT{f}{g}} , \posfInv{\cutTopp{f}{g}}\right)
\end{align*}
or for some edge $g$ of $C$ with $g \neq f$.
\end{definition}

Note that we indicate the width of $f$ by $b(f)$ rather than $w(f)$
because it is equal to the area of the boundary of the subset of $G$
defined by the boundary vertices in either of the subtrees of $C$
defined by removing $f$. This notation is also in line with the notation
in~\cite{thingraphs}.
A
notion of width for a tree position, analogous to the width of a TILO
ordering, is created by placing the widths of all of the tree edges into a
multiset and sorting in nonincreasing order. Tree widths can then be 
compared by a lexicographic ordering of the widths.
In Figure~\ref{fig:treeposition} the edges of graph $G$ have weight one and the edges
of the tree positions are denoted with their width.  
The width of $(C,\position)$ is (3,3,3,3,2,2,2) and
the width of $(C',\position')$ is (3,3,3,2,2,2,2).  Since the width of $(C',\position')$ is less than
the width of $(C,\position)$, we say that tree position $(C',\position')$ is \emph{thinner} than tree position
$(C,\position)$.
As with
linear TILO, we would like to find tree positions whose widths are as
low as possible. To do this, a \emph{branch shift} operation that is analogous to the
TILO shift is defined. A branch shift is described by two equivalent
operations: first an edit operation of removing and
reattaching a vertex (see Figure \ref{fig:bsEdit}), 
and second a cyclic shifting of edges (see Figure \ref{fig:bsCycle}).

\begin{figure}[!t]
\begin{center}
\includegraphics{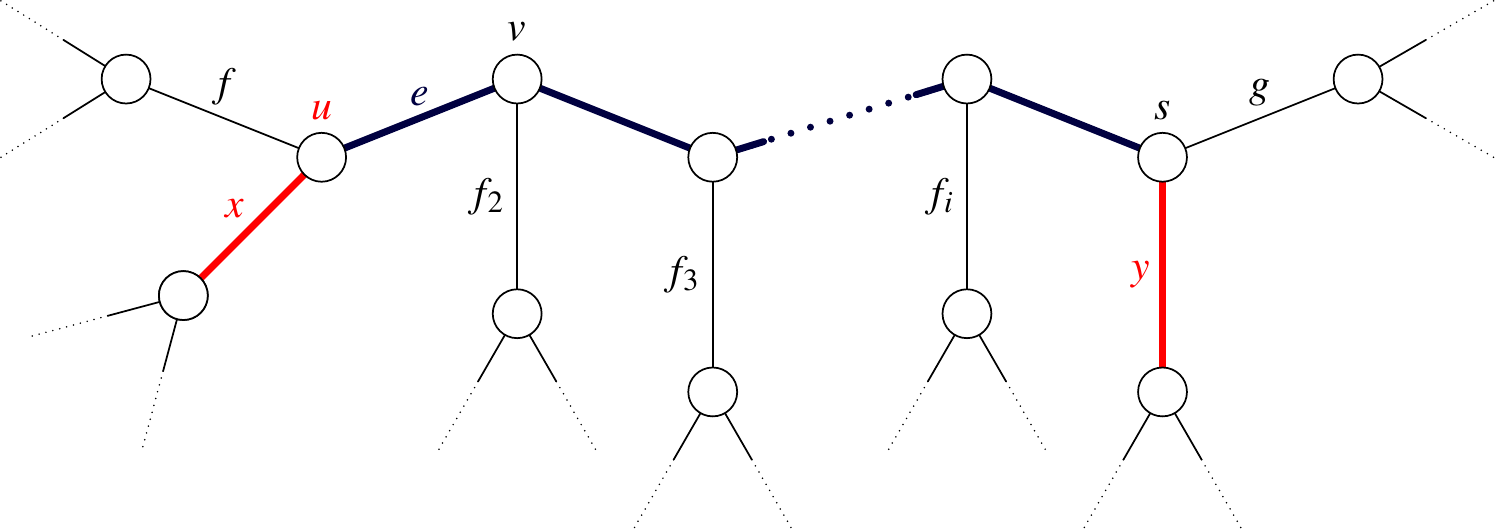}

\vspace*{2ex}

\includegraphics{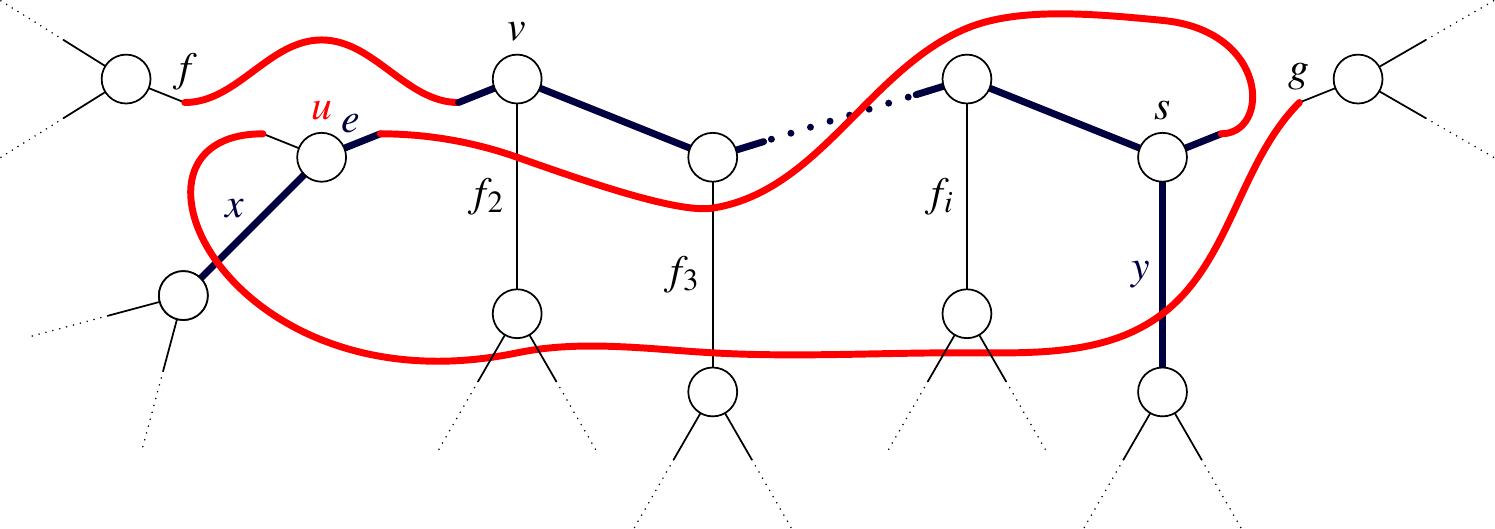}

\vspace*{2ex}

\includegraphics{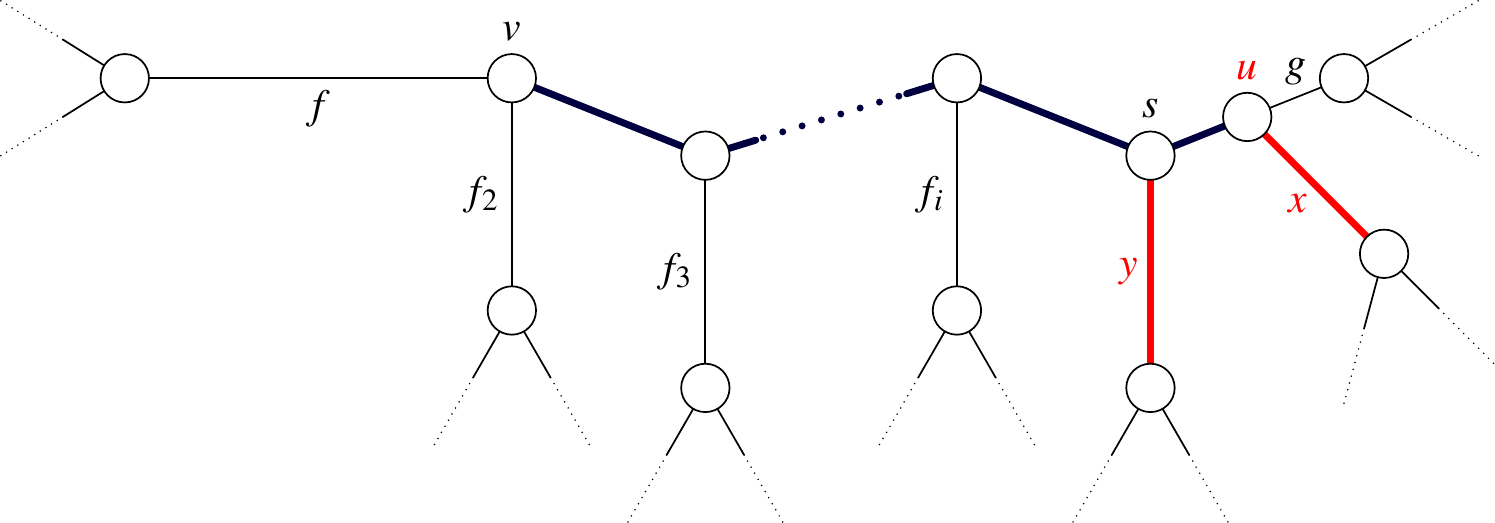}
\end{center}

\caption{Branch shift of $x$ to $y$ as an edit operation}
\label{fig:bsEdit}
\end{figure}

Given a \vt $C$, let $x$ and $y$ be edges of $C$ that do
not share a common endpoint.  Since $C$ is a tree, there is a unique
simple path from $x$ to $y$.  Let vertex $u$ be the endpoint of $x$
that is shared with the first edge of the path, $e$.    Let edge $f$ be
other edge sharing vertex $u$.  Removing vertex $u$ from
$C$ is the operation of disconnecting edge $e$ from vertex $v$ (the
other endpoint of $e$), disconnecting edge $f$ from vertex $u$, and
attaching the free end of $f$ to vertex $v$.
Let vertex $s$ be the common endpoint of $y$ and the last edge of the
path.  Let edge $g$ be adjacent to $y$ and the last edge of the path.  Adding 
vertex $u$ to $C$ is the operation of disconnecting edge $g$ from from
vertex $s$, attaching the free end of edge $e$ to vertex $s$, and attaching
the free end of edge $g$ to vertex $u$. See Figure \ref{fig:bsEdit}
for an example in which the original tree position is at the top, the middle is the tree position
after modifying the edges but not moving any vertices in the display, and the bottom is the same 
tree position but with vertex $u$ moved in the display.

An alternative view of a branch shift operation is that of a cyclic shift
of the edges adjacent to the path from $x$ to $y$.  In this view, $x$,
$y$, and all of the non-path edges sharing a vertex in the interior of the
path are disconnected.  Edge $x$ is reattached at edge $y$'s old
location.  The rest of the edges are reattached one vertex over in the
direction of $x$.  See Figure \ref{fig:bsCycle} for an example.

\begin{figure}[!t]
\begin{center}
\includegraphics{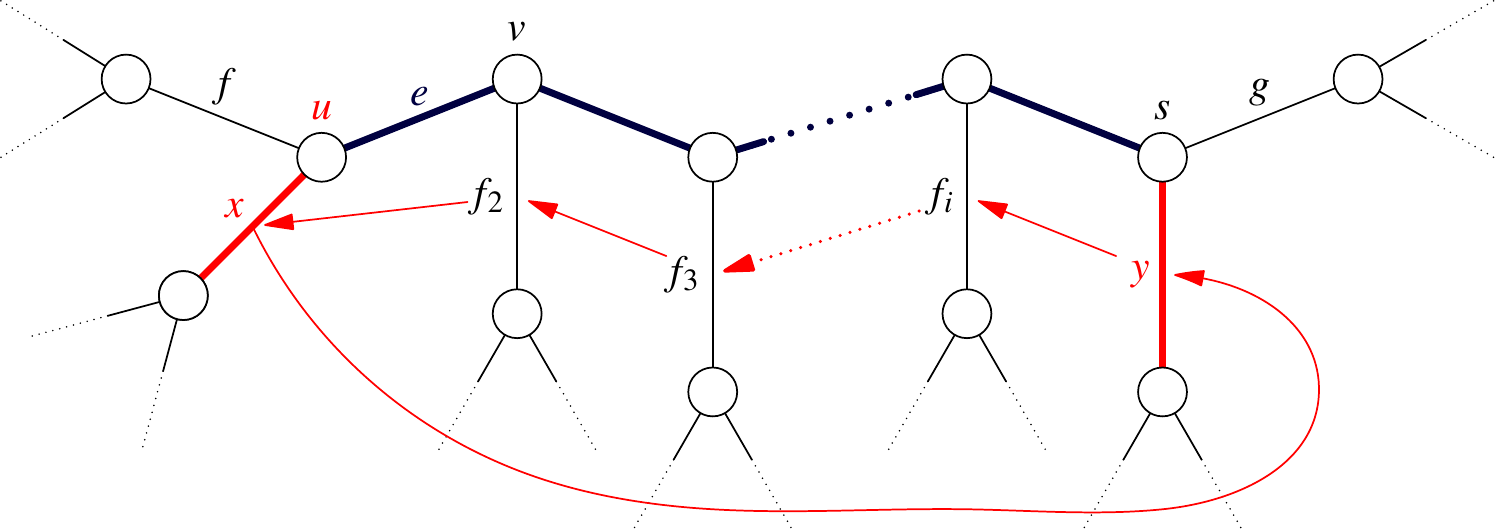}

\vspace*{5ex}

\includegraphics{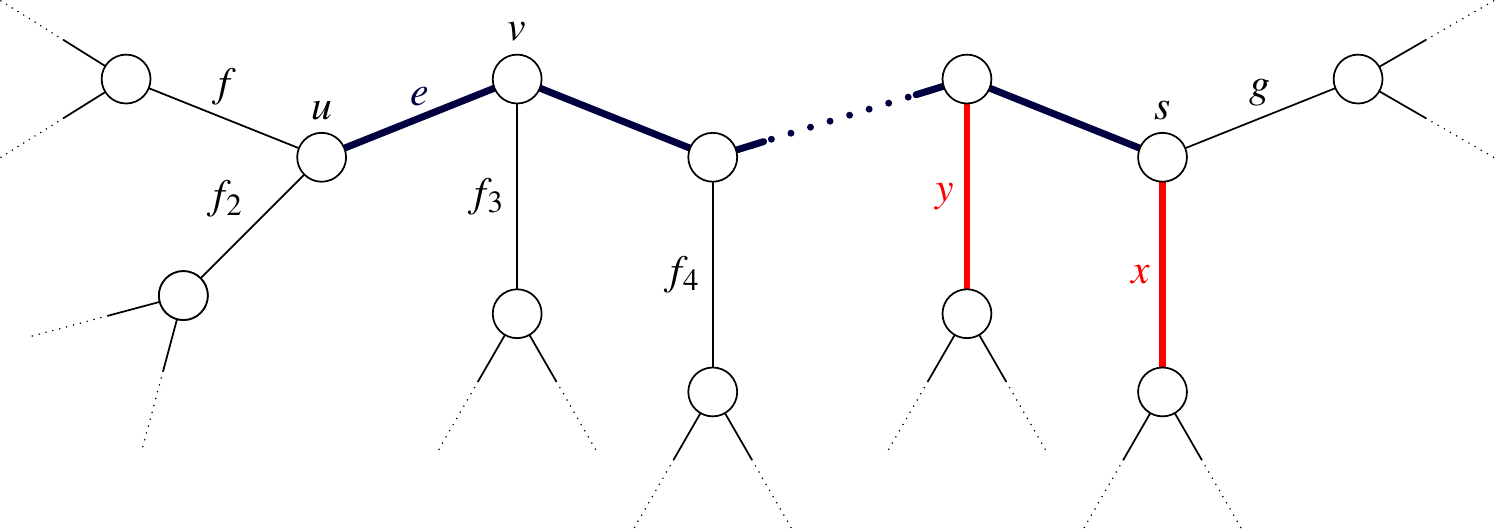}

\end{center}

\caption{Branch shift of $x$ to $y$ as a cyclic shift of branches}
\label{fig:bsCycle}
\end{figure}

The difference between viewing the branch shift as a cyclic shift of
edges and the edit operation is just a relabelling of edges and
internal nodes.  The cyclic shift of edges will be used when making a
connection to shift operations on TILO linear orderings.  A typical
software implementation will use the edit operation.  Applying the
branch shift move on $C$ creates the new tree $C'$ and is described as
\emph{shifting $x$ to $y$}.  Note that shifting $y$ to $x$ produces a
different tree than shifting $x$ to $y$.

If we perform a branch shift on a \vt $C$, the boundary vertices of $C$
are not touched. There is thus a canonical map from the boundary
vertices of $C$ to the boundary vertices of $C'$. If $(C, \position)$ is a
tree position for a graph $G$ then composing $\position$ with this
canonical map defines a one to one map $\phi'$ from the vertices of
$G$ to the boundary vertices of $C'$, which in turn defines a new tree
position for $G$.  In this paper, $C$ and $C'$ will share the same set
of boundary vertices and the canonical map is just the identity map,
\position = $\position'$.

\begin{definition}
The tree position $(C', \phi)$ is a \emph{branch shift} of $(C, \phi)$ 
when $C'$ is a branch shift of $C$.
\end{definition}

The new tree position defined by a branch shift has its own width,
which may be higher or lower than the original.

\begin{definition}
A tree position $(C, \position)$ is \emph{weakly reducible} if
there is a branch shift that produces a new tree position with
strictly lower width. Otherwise,  $(C, \position)$ is
\emph{strongly irreducible} or a \emph{thin tree position}.  
\end{definition}

\section{Comparing TILO tree position orderings to TILO linear orderings}
\label{sect:compare}

Initially, the idea of tree position may appear to be completely
different from the linear orderings used in the original TILO
algorithm. However, as we will see in this section, it is in fact a
very natural generalization.

The TILO algorithm as described in~\cite{thingraphs}
begins with a linear ordering of the vertex set $V$,
i.e.\ a one to one function from $V$ onto the set $[0,N-1] =
\{0,1,\ldots,N-1\}$, where $N$ is the number of vertices in $V$. We
will call this function a \emph{TILO ordering}. A TILO ordering of $V$
defines a sequence of subsets $A_i = \{v \in V\ |\ o(v) \leq i\}$.
This, in turn, defines a sequence of boundary widths $b_i =
w(A_i,\overline{A}_i)$ where $\overline{A}_i = \{ v \in V | v \notin
A_i\}$ is the complement of $A_i$.
Given a TILO ordering $o : V \rightarrow [0,N-1]$, we say that a
second TILO ordering $o'$ is the result of a \emph{shift} on $o$ from
$a$ to $b$ if $o'$ of the composition of $o$ with a cyclic permutation
of the block $[a,b]$ of consecutive numbers in $[0,N-1]$ that sends
$a$ to position $b$. For example, for $N = 8$, a shift of $2$ to $5$
takes the sequence $0,1,2,3,4,5,6,7$ to $0,1,3,4,5,2,6,7$. The shift
from $5$ to $2$ takes the initial sequence to $0,1,5,2,3,4,6,7$.
The width of a TILO ordering is the multiset of widths $b_i$ sorted in
nonincreasing order.  Lexicographic order is used to compare widths.

\begin{definition}
\label{def:stronglyirred}
A TILO linear ordering $o$ is called \emph{weakly reducible} if it satisfies
the reduction criteria in Lemma~\ref{lem:TILOreduction}.
A TILO linear ordering is called \emph{strongly
irreducible} if it is not weakly reducible.
\end{definition}

By definition, a weakly reducible ordering admits a shift that decreases
its width. The TILO algorithm performs the series of shifts determined by
Lemma~\ref{lem:TILOreduction} that reduce the width until it finds a strongly
irreducible ordering.  (Because there are a finite number of orderings, at
least one must be strongly irreducible.) 
To present Lemma~\ref{lem:TILOreduction}, we first need to define
slope $s_{i,k}$ with respect to order $o$ is  as the sum of the edges from vertex
$v_{o(k)}$ to other vertices in $\overline{A}_i$ minus the sum of edges from 
$v_{o(k)}$ to other vertices in $A_i$:  
\begin{gather} \label{eqn:slopeTILO}
s_{i,k} = 
w\left(v_{o(k)}, \overline{A}_i \setminus \{v_{o(k)}\}\right) -
w\left(v_{o(k)}, A_i \setminus \{v_{o(k)}\}\right) 
\end{gather}
where $\setminus$ means set difference.  The adjacency $a_{i,j}$ with
respect to order $o$ is the edge weight between the $i$-th and $j$-th
vertices of order $o$: $a_{i,j} = w( v_{o(i)}, v_{o(j)})$.  The
following lemma is a slight rewording of \cite[Lemma 3]{thingraphs}
for the case of shifting a vertex earlier in a ordering with a vertex
later in the ordering.
\begin{lemma}
\label{lem:TILOreduction}
A TILO linear ordering $o$ is weakly reducible if for some $i$ and $k$
such that $k < i$ the following conditions hold:
\begin{gather*}
s_{i,k} > 0, \\
b_{t-1} \leq b_{t}\ \text{ for all }\ k < t \leq i, \text{ and}\\
s_{i,k} - s_{i,i+1} - 2 a_{k,i+1} > 0.
\end{gather*}
In this case, shifting $k$ to $i$ reduces the width of $o$. 
\end{lemma}
The analogous result for when $k > i$ also holds but is not needed for
this paper.

To make a connection between a linear ordering and a path view of a
tree ordering, the following structures are defined.  Recall that a
\emph{partition} $\mathcal{P}$ of a set $A$ is a collection of
pairwise disjoint subsets $\mathcal{P} = \{P_0,\ldots,P_k\}$ of $A$
whose union is $A$. In other words, each element of $A$ is in exactly
one subset $P_j$.

\begin{definition} \label{def:quotientgraph}
Given a graph $G=(V,E)$ and a partition $\mathcal{P}$ of $V$, 
the \textit{quotient graph} $Q = Q(G, \mathcal{P})$ is a weighted
graph whose vertices correspond to elements $P_j$ of $\mathcal{P}$ and
whose edges are defined as follows: 
if $v_s \in P_i$, $v_t \in P_j$, $i \neq j$, and $(s,t) \in E$ 
then there is an edge $(i,j)$ in $Q$
with weight $w(P_i,P_j)$. 
\end{definition}

Note that the quotient graph only accounts for edges that go between
different sets in the partition and ignores edges between vertices
that are in the same subset $P_j$.  The quotient graph corresponds to
a coarsening of the original graph in multilevel graph
algorithms~\cite{multilevelGraph,karypis1998parallel}.


A tree position $(C, \phi)$ for a graph $G$ can be used to define a
number of different partitions of $G$. This paper is interested in a
partitioning related to the branch shift operation with the following
definition.

\begin{definition} \label{def:bsqg}
Given a tree position $(C, \position)$ for a graph $G$ and any pair of nonadjacent
edges $x$ and $y$, we can construct a partition of the vertices of $G$ using 
the following approach.  Denote the edges in the unique simple path from $x$ to $y$
as $p_t$ with $t=1$ for the edge adjacent to $x$ and $t=i$ for the
edge adjacent to $y$.  Label the edges sharing a common vertex with
adjacent path edge pairs as follows: $p_0$ for pair $x$ and $p_1$, $p_{i+1}$
for pair $y$ and $p_i$, and $f_t$ for pair $p_{t-1}$ and $p_{t}$ with
$2 \leq t \leq i$.  See Figure \ref{fig:bsqg} for an example.  Define the sets $P_t$ as
\begin{align*}
P_0 &= \posfInv{\cutTopp{p_0}{y}},\\
P_1 &= \posfInv{\cutTopp{x}{y}},\\
P_t &= \posfInv{\cutTopp{f_t}{y}}\ \text{ for } 2 \leq t \leq i,\\
P_{i+1} &= \posfInv{\cutTopp{y}{x}},\\
P_{i+2} &= \posfInv{\cutTopp{p_{i+1}}{y}}.
\end{align*}
We will call the quotient graph created using the partition
$\mathcal{P}=\set{P_0,\dots,P_{i+2}}$ the \emph{branch shift quotient
  graph} of $G$ induced from $x$, $y$, and $(C,\position)$.
\end{definition}

\begin{figure}[!t]
\includegraphics{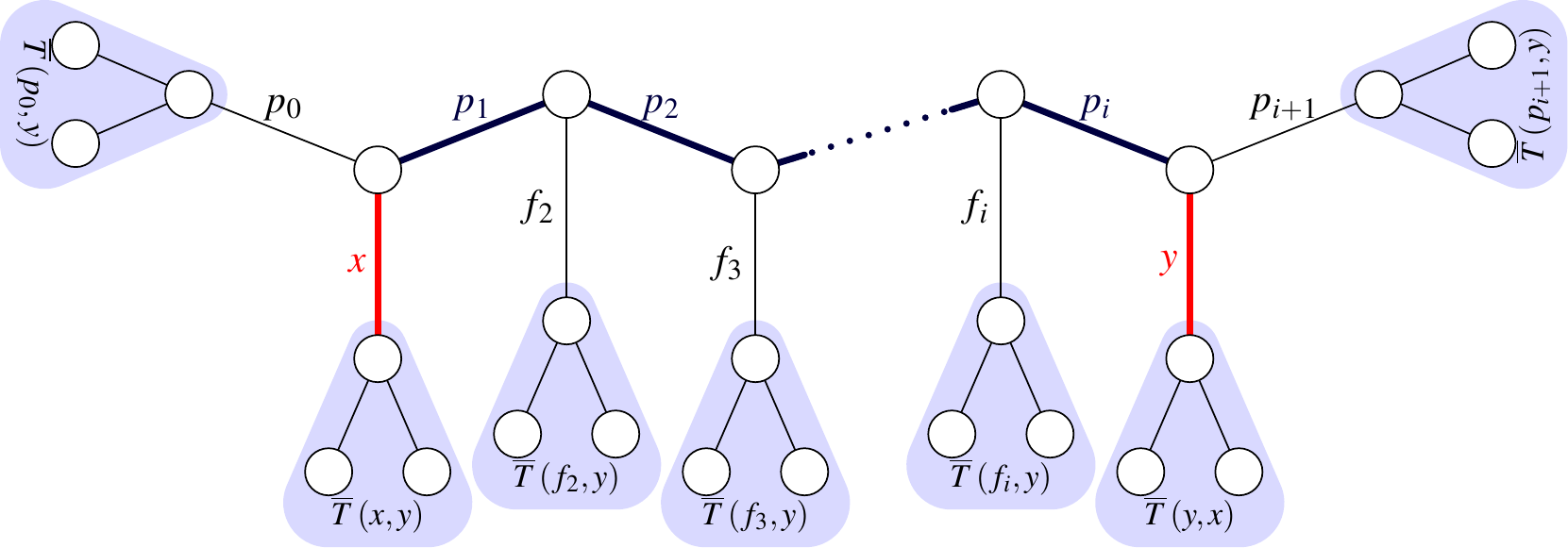}

\caption{Before cyclic branch shift of $x$ to $y$}
\label{fig:bsqg}
\vspace*{5ex}
%
\includegraphics{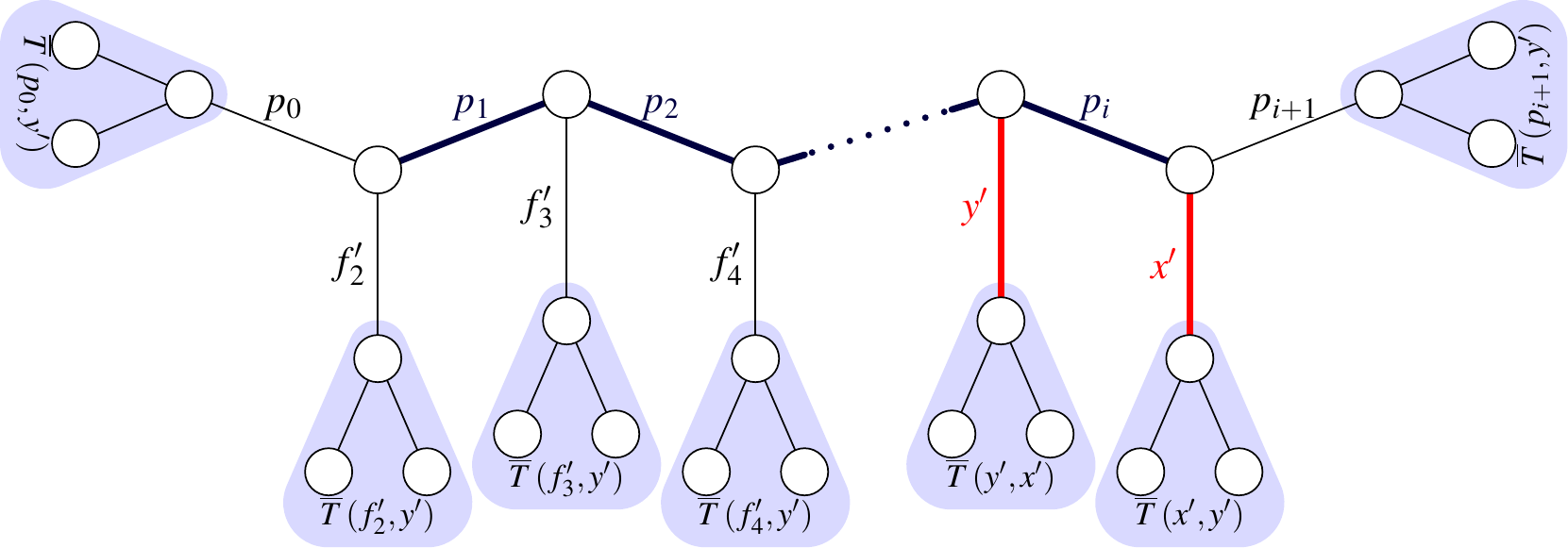}
\caption{After cyclic branch shift of $x$ to $y$}
\label{fig:shifted}
\end{figure}

The partitions $P_j$ of a branch shift quotient graph are
disjoint. The only vertices not in the union of subtrees cut from $C$
are the vertices along the path from $x$ to $y$, which must be
interior vertices.  Hence all of the boundary vertices are covered by
the subtrees and the mapping back to $G$ covers all of the vertices of
$G$.  Thus $\mathcal{P}$ is indeed a partition of $V$.

The smallest number of partitions in a branch shift quotient graph of
G induced from $x$, $y$, and $(C,\position)$ is four.  This occurs
when the path from $x$ to $y$ is just a single edge.  The largest
possible number of partitions in a branch shift quotient graph is
$|V|$, the number of vertices in $G$.  This occurs when the path from
$x$ to $y$ contains every interior vertex of $C$.  Each partition is
then a singleton set.  

TILO uses a linear ordering of vertices.  A natural ordering of the
vertices of $Q$ is to use an identity map.  That is, if $u_j$ is the
vertex of $Q$ corresponding to subset $P_j$ then it is in the $j$th
position of the linear order.  We demonstrate a connection between
this TILO ordering and the tree position with a series of Lemmas.  

\begin{lemma} \label{lem:widthEquiv}
Given a branch shift quotient graph $Q$ induced from edges $x$,
$y$, and tree position $(C,\position)$, the $t$-th width
$b_Q(t)$ using the identity map ordering on $Q$ is equal to $b(p_t)$, the
width of edge $p_t$ on the tree position $(C,\position)$ where path
edges $p_t$ are defined in Definition \ref{def:bsqg}.
\end{lemma}

\begin{proof}
The $t$-th width using the identity map linear ordering of vertices
is the sum of edge weights for edges $(j,k) \in Q$ such that $v_j$ is 
in set $\{v_0,\ldots,v_t\}$ and $v_k$ is in set
$\{v_{t+1},\ldots,v_{i+2}\}$. Since edge weight $(j,k)$ is defined as
weight $w(P_j,P_k)$, this becomes
\begin{equation*}
b_Q(t) = w\left( \bigcup_{0 \leq j \leq t} P_j , \bigcup_{t < k \leq
i+2} P_k\right).
\end{equation*}
Recall that the boundary at an edge in $C$ can be defined in terms of 
all the edges of $G$ that pass through it, $b(p_t) = \sum_{g \in
R(p_t)} w(g)$, or in terms of the edges between the sets of vertex images
of the subtrees created by cutting  at $p_t$,
\begin{equation*}
b(p_t)=
w\left(\posfInv{\cutT{p_t}{p_0}},\posfInv{\cutTopp{p_t}{p_0}}\right). 
\end{equation*}
Edge $p_t$ is not contained in any subtree used to define the
partitions in Definition \ref{def:bsqg} therefore each subtree (and
corresponding partition) must be completely on one side of the cut of
edge $p_t$. 
Thus  
\begin{align*}
P_j &\subseteq \posfInv{\cutT{p_t}{p_0}} \ \text{ for } \   0\leq j \leq t \  
\text{and}\\
P_k &\subseteq \posfInv{\cutTopp{p_t}{p_0}} \  \text{ for } \   t < k \leq i+2.
\end{align*}
Since subsets $P_j$, $P_k$ form a partition and
\begin{align*}
\posfInv{\cutT{p_t}{p_0}} \ \bigcap \ \posfInv{\cutTopp{p_t}{p_0}} = \emptyset
\end{align*}  
then
\begin{align*}
\posfInv{\cutT{p_t}{p_0}} \subseteq  \bigcup_{0 \leq j \leq t} P_j 
\quad \text{ and }   \quad 
\posfInv{\cutTopp{p_t}{p_0}} \subseteq  \bigcup_{t < k \leq i+2} P_k.
\end{align*}  
With the two pair of sets being equal, the weights are also equal:

\begin{align*}
b_Q(t) &= w( \bigcup_{0 \leq j \leq t} P_j , \bigcup_{t < k <= i+2} P_k)\\
 &= w(\posfInv{\cutT{p_t}{p_0}},\posfInv{\cutTopp{p_t}{p_0}}) \\
 &= b(p_t). 
\end{align*}
\end{proof}

A branch shift on tree position $(C,\position)$ creates 
a new tree position $(C',\position)$ for $G$. Let the 
branch shift be the shifting of edge $x$
to edge $y$ for any pair $x$ and $y$ of nonadjacent edges of $C$.
Let $Q$ be the branch shift quotient graph of $G$ induced from $x$,
$y$, and $(C,\position)$.  
Using the notation from Definition \ref{def:bsqg}, the before and after
views of a branch shift are presented in  Figures  \ref{fig:bsqg} and
 \ref{fig:shifted}.
Note that the edges $p_0$ to $p_{i+2}$ are not modified  nor are the
subtrees opposite the cut of edges $x'$, $y'$, $f_2'$ to $f_i'$.
Let $Q'$ be the branch shift quotient graph of $G$ induced from $f_2'$,
$x'$, and $(C',\position)$.  

\begin{lemma}
\label{lem:branchmovequotient}
The branch shift quotient graph $Q'$ defined above is equal to the
original branch shift quotient graph $Q$.
\end{lemma}

\begin{proof}
The each $P_j$ of the partition used by the original branch shift
quotient graph is defined in Definition~\ref{def:bsqg} using the
subtree from the far side cut of an edge adjacent to, but not on, 
the path between $x$ and $y$.  The branch shift operation cuts these
edges and then reattaches them to vertices on the path between $x$
and $y$.  This operation does not touch the subtree on the far side
cut of the edge.  Thus the subsets $P_k'$ used by $Q'$ are just a
relabelling of original subsets; explicitly, $P_0' = P_0$,
$P_{i+2}'=P_{i+2}$, 
$P_{i+1}'=P_{1}$, and 
$P_{t}'=P_{t+1}$ for $1\leq t \leq i$.
Therefore, the partition defined after the branch shift is the
same as the partition before the shift and hence the quotient graphs
are equal.
\end{proof}

Since the quotient graphs are equal, i.e.\ $Q' = Q$, we can compare
the linear TILO orderings on $Q$ defined by the tree positions $(C, \phi)$
and $(C', \phi)$.

\begin{lemma}
\label{lem:branchshiftTILO}
The TILO ordering on $Q$ defined by $(C',\phi)$ is the result of
starting with the ordering defined by $(C,\phi)$ and cyclicly shifting
the vertex at position $1$ to position $i+1$.  
\end{lemma}

\begin{proof}
 Let $u_j$ and $u_j'$ be the vertices of $Q$ and $Q'$,
respectively.  As noted in the proof of Lemma~\ref{lem:branchmovequotient}, the
subsets $P_k'$ used by $Q'$ are relabelling of the subsets $P_j$ used
by $Q$. This results in the following relabelling of vertices:
$u_0' = u_0$, $u_{i+2}'=u_{i+2}$, $u_{i+1}'=u_{1}$, and 
$u_{t}'=u_{t+1}$ for $1\leq t \leq i$.  The identity map linear
ordering of $Q'$ is 
$(u_0',u_1',u_2',\ldots, u_i',u_{i+1}',u_{i+2}')$.
Substituting $Q$ vertices into this sequence yields the vertex order
$(u_0,u_2,u_3,\ldots, u_i,u_{i+1},u_1,u_{i+2})$.
This is precisely the result of applying a cyclic shift of vertex
$u_1$ to vertex $u_{i+1}$ of the identity map linear ordering of $Q$.  
\end{proof}

\begin{corollary}
\label{coro:wrquotient}
A tree position $(C,\phi)$ of a graph $G$ is weakly reducible if 
there is a pair of edges in $C$ such that the TILO ordering of
the induced branch shift quotient graph $Q$ of $G$ is weakly
reducible.
\end{corollary}

\begin{proof}
Note that if $(C',\phi)$ is the result of a branch shift of edge $x$
to edge $y$ on a tree position $(C, \phi)$ then the width of the
edges that are not along the path from $x$ to $y$ do not change. The
only widths that change are along the path from $x$ to $y$,
and by Lemma~\ref{lem:branchshiftTILO}, the way they change is
determined by the changes to the induced orderings of the quotient
graph. 

Assume there is a path in which the induced ordering on the quotient
graph is weakly reducible, i.e.\ admits a shift that reduces its
width. This TILO shift determines a branch shift of the tree position.
Under this branch shift, the widths of the edges outside the path do
not change. Along the path, some of the widths may increase, but at
least one width strictly decreases and no width increases to a value
greater than or equal to the highest width that strictly decreases.
The same condition is true of the overall multiset of widths of the
tree position, so the width of the resulting tree position is strictly
less than the original width.

\end{proof}


\section{Reduction criteria}
\label{sect:reduction}

Corollary~\ref{coro:wrquotient} suggests a straightforward way to find
strongly irreducible tree positions of graphs: take all pairs of
edges, form the induced branch shift quotient graph, look for valid
TILO shifts, and apply the corresponding branch shifts on the tree
position.  Repeat this process until there are no valid TILO shifts,
and then perform a more comprehensive check to ensure that the tree
position is indeed strongly irreducible.  However, such an approach
involves a great deal of redundant computation. In particular, if two
paths overlap along a stretch of edges, then there will be TILO moves
in the two different quotients that correspond to the same branch
shifts.  In a naive implementation these will be checked multiple
times.  We therefore define criteria for weak reducibility of a tree
position that are determined by the quotient graphs but do not
require explicitly computing them.

\begin{definition} \label{def:adjwSlope}
Given a tree position $(C,\phi)$ and a pair of edges $e, g$ of $C$,
define the \emph{adjacency weight}
\begin{equation*}
a(e,g) = w\left(R(e) \cap R(g)\right) = w\left(\posfInv{\cutTopp{e}{g}},\posfInv{\cutTopp{g}{e}}\right)
\end{equation*}
to be the sum of the weights of the edges of $G$ that pass through both
$e$ and $g$ or equivalently the sum of the weights of edges of $G$ between
boundary vertices in \cutTopp{e}{g}  and boundary vertices in \cutTopp{g}{e}.

Define the \emph{slope} 
\begin{align*}
s(e,g) &= w\left(R(g) \setminus R(e)\right) -  w\left(R(e) \cap R(g)\right)\\
       &= w\left(\posfInv{\cutT{g}{e}} \setminus \posfInv{\cutTopp{e}{g}}, \cutTopp{g}{e}\right) - w\left(\posfInv{\cutTopp{e}{g}},\posfInv{\cutTopp{g}{e}}\right)  
\end{align*}
to be the sum of the weights of the edges of
$G$ that pass through $g$ but not through $e$  minus 
the sum of the weights of edges of $G$ that
pass through both edges.
The slope $s(e,g)$ is the amount that the width of edge $e$ changes 
if edge $g$ is disconnected from the tree and reattached on the far side
of $e$.
\end{definition}

 From the definitions, it follows
that $$a(e,g) + s(e,g) = w(R(g) \setminus R(e)).$$
From basic set theory, 
$$R(g) = (R(g) \cap R(e)) \cup (R(g) \setminus R(e))$$ 
where the sets $R(g) \cap R(e)$ and $R(g) \setminus R(e)$ are
disjoint. Thus we have $b(g) = a(e,g) + a(e,g) + s(e,g)$. Rearranging
this gives us an alternate way to calculate $s(e,g)$:

\begin{lemma}
\label{lem:equivslopedef}
The slope $s(e,g)$ can also be calculated by the equation
\begin{align*}
s(e,g) = b(g) - 2a(e,g).
\end{align*}
\end{lemma}

If we think of the vertices $u_e$ and $u_g$ in a quotient
graph that correspond to the subtrees \cutTopp{e}{g} and
\cutTopp{g}{e}, respectively, then the adjacency weight $a(e,g)$
is the weight of edge $(e,g)$ of the quotient graph.
If $u_e$ and $u_g$ are vertices in a branch shift quotient graph
then the slope $s(e,g)$ is $\pm s_{e,g}$, the slope (as defined
in~\cite{thingraphs}) of a  linear ordering of the vertices of
the induced quotient graph.
In particular, we have the following:

\begin{lemma}
\label{lem:slopetranslation}
Given $Q$, a branch shift quotient graph of graph $G$ induced by edges
$x$, $y$, and tree position $(C,\position)$, let $e=p_i$, an edge
on the path from $x$ to $y$ and $g=f_k$, a non-path edge adjacent
to a path edge from $x$ to $y$.  Using the identity map induced linear
TILO ordering on $Q$, we find that
\begin{align*}
s(e,g)= s(p_i,f_k) = \begin{cases}
-s_{i,k} & \text{ if } k \leq i,\\
s_{i,k} & \text{ if } k > i
\end{cases}
\end{align*}
where the slope $s_{i,k}$ of a linear order is defined in
(\ref{eqn:slopeTILO}).
\end{lemma}

\begin{proof}
Using the notation from Definition \ref{def:bsqg} and Lemma \ref{lem:widthEquiv}, the 
path edge $p_i$ corresponds to 
the location between quotient graph vertices $u_i$ and $u_{i+1}$ in the identity map linear ordering.
The edge $f_k$ corresponds to vertex $u_k$ of the quotient graph.
Recall that $s_{i,k}$ is defined as the sum of the weights of edges from vertex
$v_k$ to vertices with indices strictly greater than $i$ minus the sum of the
edge weights from vertex $v_k$ to vertices with indices less than or equal to
$i$.  
\begin{align*}
s_{i,k} &= \sum_{\substack{i < j\\j \neq k}} w(k,j) - \sum_{\substack{0 \leq j \leq i \\j \neq k}} w(k,j)\\
        &= w \Bigl( P_k,\bigcup_{\substack{i < j\\j \neq k}} P_j \Bigr) - w\Bigl(P_k, \bigcup_{\substack{0 \leq j \leq i \\j \neq k}} P_j\Bigr) 
\end{align*}

Cutting at edge $p_i$ and set subtracting $P_k$ creates the sets corresponding to the two unions in the above equation.
When $k>i$ this is
$\bigcup_{\substack{i < j\\j \neq k}} P_j = 
\posfInv{\cutT{p_i}{f_k}} \setminus \posfInv{\cutTopp{f_k}{p_i}}$
and $\bigcup_{\substack{0 \leq j \leq i \\j \neq k}} = \posfInv{\cutTopp{p_i}{f_k}} \setminus \posfInv{\cutTopp{f_k}{p_i}}
=\posfInv{\cutTopp{p_i}{f_k}}$
with
$P_k = \cutTopp{f_k}{p_i}$.  Using
$\posfInv{\cutT{p_i}{f_k}} \setminus \posfInv{\cutTopp{f_k}{p_i}} 
= \posfInv{\cutT{f_k}{p_i}} \setminus \posfInv{\cutTopp{p_i}{f_k}} $ and substituting into the above equation yields
\begin{align*}
s_{i,k} &= w\left(\cutTopp{f_k}{p_i}  ,\posfInv{\cutT{f_k}{p_i}} \setminus \posfInv{\cutTopp{p_i}{f_k}} \right) - 
w\left( \cutTopp{f_k}{p_i}, \posfInv{\cutTopp{p_i}{f_k}} \right),\\
 &= w\left(\cutTopp{g}{e}  ,\posfInv{\cutT{g}{e}} \setminus \posfInv{\cutTopp{e}{g}} \right) - 
w\left( \cutTopp{g}{e}, \posfInv{\cutTopp{e}{g}} \right)
\end{align*}
which is the definition of slope $s(e,g)$.  When $k\leq i$, then the unions swap their correspondence to \cutT{p_i}{f_k}
and \cutTopp{p_i}{f_k} yielding $s_{i,k} = -s(e,g)$. 
\end{proof}

\begin{lemma} \label{lem:treeReducible}
A tree position $(C, \phi)$ is weakly reducible if for
some pair of nonadjacent edges $x$ and $y$ the following conditions hold: 
\begin{gather}
s(p_i,x)<0,   \\
b(p_{t-1}) \leq b(p_t) \ \text{ for all }\  1<t \leq i, \text{ and }\\
s(p_i,x) + s(p_i,y) +2a(x,y) < 0  \label{eqn:treduceB}
\end{gather}
where the edges of the simple path connecting $x$ to $y$ is denote by $p_t$ for $t$ from 1 to $i$.
Moreover, if this is the case then shifting $x$ to $y$ strictly
reduces the width of the tree position.  
\end{lemma}

\begin{proof}
Let $Q$ be the branch shift quotient graph of $G$ induced by $x$, $y$,
and $(C,\position)$ with an identity map linear TILO ordering of $Q$'s
vertices $u_t$.  In terms of Definition \ref{def:bsqg}, $x$ is
associated with $u_1$ and $y$ is associated with $u_{i+1}$.  In terms
of Lemma~\ref{lem:TILOreduction}, $k=1$ as $u_1$ is being shifted to
$u_{i+1}$.  The resulting weakly reducible conditions are a
nonincreasing sequence of boundaries from $u_1$ to $u_i$, $s_{i,1}>0$,
and $s_{i,1} - s_{i,i+1} - 2a_{1,i+1}>0$.  By
Lemma~\ref{lem:slopetranslation}, we have $s_{i,1} = -s(p_i,x)$ and
$s_{i,i+1} = s(p_i,y)$.  By Definitions \ref{def:quotientgraph} and
\ref{def:adjwSlope}, we have $a_{1,i+1} = a(x,y)$.  By
Lemma~\ref{lem:widthEquiv}, we have $b_Q(t)$ = $b(p_t)$.  Therefore by
substituting, we find that the TILO ordering of the quotient graph $Q$
is weakly reducible if 
\begin{gather*}
-s(p_i,x) > 0,\\
b(p_{t-1}) \leq  b(p_t), \ \text{ and }\\
-s(p_i,x) - s(p_i,y) - 2a(x,y) > 0.
\end{gather*}
By Corollary~\ref{coro:wrquotient}, the tree position is weakly
reducible if this condition holds.  Multiplying the slope inequalities
by $-1$ gives us the condition stated above.  
\end{proof}

Condition (\ref{eqn:treduceB}) of the lemma gives the amount by which
the width of edge $p_i$ changes due to shifting $x$ to $y$,
\begin{equation} \label{eqn:widthChange}
 b^{'}(p_i) = b(p_i) +s(p_i,x) + s(p_i,y) +2a(x,y)
\end{equation}
where $b^{'}(p_i)$ is the width after the branch shift operation.

\section{Properties of thin tree positions useful for clustering and
  implementation}

The properties of a thin tree position of a graph $G=(V,E)$ can be
used to find \emph{pinch clusters} of $G$ (see~\cite[Definition
  1]{thingraphs}).  A subset $A \subset V$ is a pinch cluster if for
any sequence of vertices $v_1,\dots, v_m$, if adding $v_1,\dots, v_m$
to $A$ or removing $v_1,\dots, v_m$ from $A$ creates a new set with
smaller boundary, then for some $k<m$, adding or removing $v_1,\dots,
v_k$ to or from $A$ creates a set with strictly larger boundary.
Recall that, given a subset $A\subseteq V$, we say that the
\emph{boundary} of $A$ is $w(A, V\setminus A)$.

In linear TILO, pinch clusters are determined by local
minima in the width values defined by the TILO ordering. 
Any path in a thin tree position with local minimums of edge widths
determines pinch clusters of the quotient graph by cutting the graph at the
local minimum. These clusters can be refined to create pinch clusters
of the original graph by creating a linear ordering of vertices
compatible with the tree position and checking for any TILO
shifts. Limiting the search for refinement shifts and approximating
the pinch ratios by effective bounds calculated from a tree position
will be explored in future papers.  Some pinch clusters of $G$ can be determined
without futher refinement.  These occur at locations defined as follows.

\begin{definition}
Given a graph $G$ and a tree position $(C, \position)$ for $G$, we say that
an edge $e$ of $C$ is a \emph{local minimum} if each endpoint of $e$
is shared with an edge of $C$ that has strictly greater width than $e$
and whose second endpoint is not a boundary vertex.  
\end{definition}

When an edge is a local minimum of a tree position then it is a local minimum of every path 
through it in the tree.  Thus the edge determines pinch clusters in every quotient graph induced
by a path through the edge.  The next theorem shows that a local minimum edge finds a pinch cluster
of the original graph.

\begin{theorem}
If $e$ is a local minimum of a strongly irreducible tree position 
$(C,\position)$ of a graph $G$ then the two subsets of $G$ defined by the
subtrees that result from cutting $C$ at $e$ are pinch clusters.
\end{theorem}

\begin{proof}
Assume $e$ is a local minimum of a tree position $(C, \position)$. 
We will prove that $e$ is a local minimum in a strongly irreducible 
TILO ordering on $G$. 

Let $\{P_1, P_2\}$ be the partition of the vertices of $G$ defined by the
local minimum $e$, i.e., $P_1 = \posfInv{\cutT{e}{g}}, P_2 =
\posfInv{\cutTopp{e}{g}}$ for some other edge $g$ in $C$.  Let $o$ be a TILO
ordering of $G$ that has minimal width among all orderings for which every
vertex of $P_1$ appears before every vertex of $P_2$. In other words, $P_1$
consists of the vertices $v_0,\ldots,v_k$, while $P_2$ consists of vertices
$v_{k+1},\ldots,v_{N-1}$ (where the subscript indicates the TILO ordering).
Note the the width $b_k$ between $v_{k}$ and $v_{k+1}$ is precisely
$b(e)$.

If $b_{k-1}>b_k$ and $b_{k+1}>b_k$ then by the second condition of Lemma~\ref{lem:TILOreduction}
it is not possible for a valid shift of a vertex between $P_1$ and $P_2$ of the ordering.  With
the initial assumption that ordering $o$ has minimal width within the partitions, this means that
$o$ is a strongly irreducible TILO ordering on $G$.
 
Let $x$ be the edge in $C$ that has $\posf{v_k}$ as an endpoint.
If $x$ and $e$ share an endpoint then $b(x)=b_{k-1}$.  Since $e$ is a local minimum, 
$b(x)>b(e)$ and thus $b_{k-1}>b_{k}$.
Assume that $x$ and $e$ do not share an endpoint and
denote the edges of the simple path from $x$ to $e$ as $p_1,\ldots,p_i$.  Let $y$ be the edge 
adjacent to $e$ and $p_i$.
Assume $b_{k-1} < b_k$.
Then $s_{k-1,k} = b_k - b_{k-1} > 0$ 
and by Lemma~\ref{lem:slopetranslation}, $s(e,x) = -s_{k-1,k} < 0$. 
For path edges $k > j$, $\cutTopp{p_k}{x} \subset \cutTopp{p_j}{x}$ and  
$a(p_k,x)\leq a(p_j,x)$. Using the slope definition in Lemma~\ref{lem:equivslopedef}, 
$s(p_j,x)=b(x)-2a(p_j,x)\leq b(x)-2a(p_k,x) = s(p_k,x)$. Since $e$ is at the end of the path,
$s(p_j,x) \leq s(e,x)$. 
This is also the same path from $x$ to $y$. 
Consider a branch shift of $x$ to $y$.
After shifting $x$ to $y$, $b^{'}(p_i)$ is $b_{k-1}$ as
cutting at $p_i$ induces the partitions of $G$ as $P_1 \setminus \set{v_k}$ and $P_2 \cup \set{v_k}$.
Shifting $x$ to $y$ strictly reduces the width $b(p_j)$
for every path edge $p_j$ since $s(p_j,x) < 0$.  No other widths are changed.
Thus shifting $x$ to $y$ reduces the width
of the tree position $(C,\position)$.  But this contradicts the original assumption that 
$(C,\position)$ is strongly irreducible.  Thus $b_{k-1}$ can not be less than $b_k$.




A symmetric argument implies that $b_{k+1}$ can not be less than
$b_k$.  In the case that $b_{k-1}=b_k$, the argument can be extended
to show that the closest nonequal boundary in each direction can not
be less than $b_k$.  So $b_k$ must be a local minimum of the TILO
ordering.  By~\cite[Theorem 4]{thingraphs}, this implies that $T_1$
and $T_2$ are pinch clusters.
\end{proof}

In addition to the conceptual advantages of thin tree position noted
in the introduction, this method is more amenable to efficient
implementation.  Width, slope, and adjacency weight are defined both
in terms of a weight of sets of edges and in terms of a weight between
sets of vertices. The edge based definition is better for accumulating
and propagating the tree position reduction calculations over sparse
matrices. The vertex based definition is better for accumulating and
propagating these calculations over dense matrices.

The following relationships between the width, slope, and adjacency
weight of edges sharing a common vertex are useful when accumulating
and propagation calculations across a tree position.  Given edges $e$,
$f$, and $g$ sharing a common interior vertex of a tree position, we have
\begin{align*}
b(e)  &= a(e,f)+a(e,g), \\
b(e) &= b(f) +b(g) -2a(f,g), \text{and}\\
  s(f,g) &= a(g,e)-a(f,g)=b(e)-b(f). \\
\end{align*}

The reduction checks of Lemma \ref{lem:treeReducible} can be done in
parallel. Since a branch shift of $x$ to $y$ does not modify the
subtrees off the path from $x$ to $y$, sets of branch shift
operations can be applied in parallel if the paths are independent
(do not intersect).



\bibliographystyle{spmpscinat}
\bibliography{thingraphs,prc,bd}   

\begin{thebibliography}{9}
\providecommand{\natexlab}[1]{#1}
\providecommand{\url}[1]{#1}
\providecommand{\urlprefix}{URL }
\expandafter\ifx\csname urlstyle\endcsname\relax
  \providecommand{\doi}[1]{DOI~\discretionary{}{}{}#1}\else
  \providecommand{\doi}{DOI~\discretionary{}{}{}\begingroup
  \urlstyle{rm}\Url}\fi

\bibitem[{Carlsson(2009)}]{Carlsson09}
Carlsson, G.: Topology and data.
\newblock Bulletin of the American Mathematical Society \textbf{46}(2),
  255--308 (2009)

\bibitem[{Gabai(1987)}]{Gabai}
Gabai, D.: Foliations and the topology of {$3$}-manifolds. {III}.
\newblock J. Differential Geom. \textbf{26}(3), 479--536 (1987).
\newblock
  \urlprefix\url{http://projecteuclid.org/getRecord?id=euclid.jdg/1214441488}

\bibitem[{Heisterkamp and Johnson(2013)}]{sdm13prc}
Heisterkamp, D.R., Johnson, J.: Pinch ratio clustering from a topologically
  intrinsic lexicographic ordering.
\newblock In: 2013 SIAM International Conference on Data Mining (SDM13), pp.
  560--568. Austin Texas (2013)

\bibitem[{Jasiński(2013)}]{BoronTrees2013}
Jasiński, J.: Ramsey degrees of boron tree structures.
\newblock Combinatorica \textbf{33}(1), 23--44 (2013).
\newblock \doi{10.1007/s00493-013-2723-6}.
\newblock \urlprefix\url{http://dx.doi.org/10.1007/s00493-013-2723-6}

\bibitem[{Johnson(2014)}]{thingraphs}
Johnson, J.: Topological graph clustering with thin position.
\newblock Geometriae Dedicata \textbf{169}(1), 165--173 (2014).
\newblock \doi{10.1007/s10711-013-9848-z}.
\newblock \urlprefix\url{http://dx.doi.org/10.1007/s10711-013-9848-z}

\bibitem[{Karypis and Kumar(1995)}]{multilevelGraph}
Karypis, G., Kumar, V.: Analysis of multilevel graph partitioning.
\newblock In: Proceedings of the 1995 ACM/IEEE Conference on Supercomputing,
  Supercomputing '95. ACM, New York, NY, USA (1995).
\newblock \doi{10.1145/224170.224229}.
\newblock \urlprefix\url{http://doi.acm.org/10.1145/224170.224229}

\bibitem[{Karypis and Kumar(1998)}]{karypis1998parallel}
Karypis, G., Kumar, V.: A parallel algorithm for multilevel graph partitioning
  and sparse matrix ordering.
\newblock Journal of Parallel and Distributed Computing \textbf{48}(1), 71--95
  (1998)

\bibitem[{Ostilli(2012)}]{CayleyTrees2012}
Ostilli, M.: {C}ayley trees and {B}ethe lattices: A concise analysis for
  mathematicians and physicists.
\newblock Physica A: Statistical Mechanics and its Applications
  \textbf{391}(12), 3417 -- 3423 (2012).
\newblock \doi{http://dx.doi.org/10.1016/j.physa.2012.01.038}.
\newblock
  \urlprefix\url{http://www.sciencedirect.com/science/article/pii/S0378437112000647}

\bibitem[{Scharlemann and Thompson(1994)}]{st:thin}
Scharlemann, M., Thompson, A.: Thin position for {$3$}-manifolds.
\newblock In: Geometric topology ({H}aifa, 1992), \emph{Contemp. Math.}, vol.
  164, pp. 231--238. Amer. Math. Soc., Providence, RI (1994).
\newblock \doi{10.1090/conm/164/01596}.
\newblock \urlprefix\url{http://dx.doi.org/10.1090/conm/164/01596}

\end{thebibliography}

\end{document}